\documentclass[11pt]{amsart}

\usepackage[T1]{fontenc}
\usepackage{lmodern}
\usepackage[a4paper,left=1in,right=1in,top=0.8in,bottom=0.8in]{geometry}
\usepackage{amsmath,amssymb,amsthm,mathtools}
\usepackage{microtype}
\usepackage[hidelinks]{hyperref}

\hypersetup{
  pdftitle={Ricci pinching on complete Kahler manifolds},
  pdfauthor={Gang Liu and Lei Ni},
  pdfsubject={Ricci pinching, Busemann exhaustion, signed Hodge heat mass, and compactness},
  pdfkeywords={Kahler manifold, Ricci pinching, nonnegative sectional curvature, positive bisectional curvature, 
  Busemann exhaustion, quadratic orthogonal bisectional curvature}
}

\newtheorem{theorem}{Theorem}[section]
\newtheorem{proposition}[theorem]{Proposition}
\newtheorem{lemma}[theorem]{Lemma}

\theoremstyle{remark}
\newtheorem{remark}[theorem]{Remark}
\numberwithin{equation}{section}
\allowdisplaybreaks[2]

\newcommand{\Ric}{\operatorname{Ric}}
\newcommand{\Rm}{\operatorname{Rm}}
\newcommand{\Vol}{\operatorname{Vol}}
\newcommand{\tr}{\operatorname{tr}}
\newcommand{\supp}{\operatorname{supp}}
\newcommand{\ddc}{dd^c}
\newcommand{\dV}{\,dV}
\newcommand{\eps}{\varepsilon}
\newcommand{\C}{\mathbb C}
\newcommand{\R}{\mathbb R}
\newcommand{\norm}[1]{\left\lVert #1\right\rVert}
\newcommand{\Ricmin}{\Ric_{\min}}
\newcommand{\Ricmax}{\Ric_{\max}}

\title[Ricci pinching on complete K\"ahler manifolds]
{Ricci pinching on complete K\"ahler manifolds}
\author{Gang Liu}
\address{Gang Liu. School of Mathematics, East China Normal University,
Shanghai, China, 20062}
\email{gliu@math.ecnu.edu.cn}
\author{Lei Ni}
\address{Lei Ni. Zhejiang Normal University, Jinhua, Zhejiang, China, 321004}
\address{Department of Mathematics, University of California, San Diego,
La Jolla, CA 92093, USA}
\email{lni.math.ucsd@gmail.com}
\date{}

\begin{document}
\begin{abstract}
We study Ricci pinching problem on complete K\"ahler manifolds. For complete
noncompact manifolds with Euclidean volume growth, we prove flatness under
nonnegative bisectional curvature and Ricci pinching, and Ricci flatness
under nonnegative Ricci curvature, quadratic curvature decay, and the same
pinching condition. In both cases, the pinching constant may tend to zero
at a logarithmic rate. Examples show that the logarithmic exponent is
sharp in complex dimension two. Second part of the paper proves that a complete K\"ahler
manifold of complex dimension at least two with positive holomorphic
bisectional curvature is compact if its  sectional curvature is
nonnegative and its Ricci curvature is uniformly pinched outside a compact
set. No volume-growth assumption or curvature upper bound is needed.
Heat smoothing of the Busemann function supplies a strictly
plurisubharmonic exhaustion with bounded gradient, so the general top
Ricci-mass theorem of Datar--Pingali--Seshadri applies. The new ingridents are integral estimates for the Ricci form. The analytic arguments use
scalar curvature gap theorems and a critical-$L^n$ vanishing theorem for
closed real $(1,1)$-forms under nonnegative Ricci curvature and nonnegative
quadratic orthogonal bisectional curvature. The Hodge-Laplace  heat eqution deformation
conserves the signed top-degree integral, including its initial
value, while norm domination and scalar heat-semigroup decay force this
integral to vanish. For semipositive initial forms this gives pointwise
vanishing of the top wedge power. 
\end{abstract}
\maketitle

\section{Introduction and statements}

The Bonnet--Myers theorem asserts that a complete Riemannian manifold with
a positive uniform lower bound for the Ricci curvature is compact. It is
natural to ask whether one can replace the uniform lower bound by a
pointwise pinching condition. More precisely, assume that the Ricci
curvature is nonnegative and
\begin{equation}\label{eq:ratio}
  \Ricmin\ge\kappa\,\Ricmax
\end{equation}
for some constant $\kappa>0$, where $\Ricmin$ and $\Ricmax$ are the smallest
and largest eigenvalues of the Ricci endomorphism. In real dimension
three, Hamilton conjectured that a complete manifold satisfying
\eqref{eq:ratio} is flat or compact. This was proved through the work of
Lott~\cite{Lott24}, Deruelle--Schulze--Simon~\cite{DSS22}, and
Lee--Topping~\cite{LT25}. A higher dimensional analogue under
$\mathrm{PIC}_1$ pinching was recently obtained in \cite{DLSS26}. For the
compactness question under positive Ricci pinching in arbitrary dimension for any Riemannian manifold,
we refer to the article by the second author \cite{Ni25} and the references therein.

In the K\"ahler setting, Shi--Tam and the second author here~\cite[Theorem~5.4]{NST01} proved
vanishing results for nonnegative closed $(1,1)$-forms whose eigenvalues
are uniformly pinched. Applied to the Ricci form, their theorem gives
flatness under nonnegative bisectional curvature, pointwise quadratic
scalar curvature decay, and the corresponding averaged decay bound.
Deng--Zhu~\cite{DZ15} established related rigidity results for expanding
and steady gradient K\"ahler--Ricci solitons; see also the discussion in
\cite{Ni25}.

The first part of this paper concerns complete K\"ahler manifolds with
Euclidean volume growth. These results are direct applications of the
integral estimates in \cite{Liu22,Liu24} and scalar curvature gap theorems.
Under nonnegative bisectional curvature, no pointwise curvature decay is
assumed. The observation is that Ricci pinching converts the finiteness
of the top Ricci integral into an integral bound for the scalar curvature.
This also allows the pinching constant to degenerate at infinity.

The second part gives a compactness theorem under positive holomorphic
bisectional curvature, with nonnegative real sectional curvature and
uniform Ricci pinching required only outside a compact set. In particular,
it includes the case of positive sectional curvature, without
assuming Euclidean volume growth. The geometric argument uses the
Busemann heat deformation of Tam and the second author ~\cite[Lemma~4.1 and the proof of
Theorem~4.2]{NT03} to construct a smooth strictly plurisubharmonic
exhaustion with bounded gradient. The finite top
Ricci-mass theorem of Datar--Pingali--Seshadri
\cite[Theorem~1.1]{DPS26} \footnote{If one applies 
the most recent arXiv preprint of these authors aied with AI, Theorem \ref{thm:main} 
holds under the positivity of the bisectional curvature only.  If Yau's conjecture held
in the full generality, Theorem \ref{thm:main} holds under nonnegativity of Ricci and quadratic orthogonal bisectional curvature.} then applies. The analytic argument uses the
closed Hodge heat evolution constructed by Tam and the second author~\cite{NT13}. We prove
that, under nonnegative Ricci curvature and nonnegative quadratic
orthogonal bisectional curvature, every closed real $(1,1)$-form whose
norm belongs to $L^n$ has zero signed top-degree integral. 
  A semipositive
such form therefore has identically vanishing top wedge power. This
provides the finite-mass rigidity theorem in Section~\ref{sec:forms}. This result shares
 the same philosophy advocated by the second author in his previous work(joint work), namely that 
 an `ergodicity principle' survives for parabalic equations 
 such as the linear heat equation, Ricci flow, Hermitian Yang-Mills flow, Hodge-Laplace heat equation, 
 Gauss curvature flow etc. The main difference is that previously the monotonicity provided the path 
 to use the large time behavior to control the initial data. 
 Now the large time asymptotic behavior is used to control the initial time data 
 via the conservation of the total mass. 
The weaker analytic curvature assumptions are kept separate from the
hypotheses of the geometric finiteness theorem.

Throughout the paper, $n$ denotes the complex dimension. All manifolds are
understood to be connected. Fix $p\in M$ and write $r(x)=d(p,x)$. We say
that $M$ has \emph{Euclidean volume growth} if
\[
  v=\lim_{R\to\infty}
       \frac{\Vol(B(p,R))}{\omega_{2n}R^{2n}}>0,
\]
where $\omega_{2n}$ is the volume of the Euclidean unit ball in real
dimension $2n$. \emph{Quadratic curvature decay} means that
\[
  |\Rm|(x)\le\frac{C}{(1+r(x))^2}
\]
for some constant $C>0$.

\subsection*{Results with Euclidean volume growth}
We first state the results with a uniform pinching constant.

\begin{theorem}\label{thm:euclidean-bk}
Let $M^n$, $n\ge2$, be a complete noncompact K\"ahler manifold with
nonnegative bisectional curvature and Euclidean volume growth. Assume that
the Ricci curvature is pinched in the sense of \eqref{eq:ratio}.
Then $M$ is flat.
\end{theorem}

\begin{theorem}\label{thm:euclidean-ric}
Let $M^n$, $n\ge2$, be a complete noncompact K\"ahler manifold with
nonnegative Ricci curvature, Euclidean volume growth, and quadratic
curvature decay. Assume that the Ricci curvature is pinched in the sense
of \eqref{eq:ratio}. Then $M$ is Ricci flat.
\end{theorem}

Ricci flatness is the best conclusion in Theorem~\ref{thm:euclidean-ric}.
For instance, the Eguchi--Hanson metric is nonflat and satisfies all the
assumptions. Ricci flatness and compactness are the alternatives in the
general question asked in \cite{Ni25}. The distinction between
Theorems~\ref{thm:euclidean-bk}
and~\ref{thm:euclidean-ric} comes from the corresponding gap theorems of
the second author~\cite{Ni12} and the first author~\cite{Liu24}.

In fact, we have the following version with a pinching constant depending
on the distance (cf.\ \cite{Ni25} for a similar result for convex
hypersurfaces).

\begin{theorem}\label{thm:logarithmic}
In Theorems~\ref{thm:euclidean-bk} and~\ref{thm:euclidean-ric}, one may
replace \eqref{eq:ratio} by
\begin{equation}\label{eq:log-pinching}
  \Ricmin(x)\ge
  \frac{c}{\bigl(\log(2+r(x))\bigr)^{1/(n-1)}}\Ricmax(x)
\end{equation}
outside a compact set, where $c>0$ is a constant. The conclusions remain
the same.
\end{theorem}

No pinching condition is needed on the compact set. In particular, the
uniform pinching condition in Theorems~\ref{thm:euclidean-bk}
and~\ref{thm:euclidean-ric} is also only required outside a compact set.

\subsection*{Compactness under positive bisectional curvature}

\begin{theorem}[Compactness]\label{thm:main}
Let $(M^n,g,J)$ be a connected complete K\"ahler manifold of complex
dimension $n\ge2$ with strictly positive holomorphic bisectional
curvature. Suppose that there are a compact set $C\subset M$ and a
constant $\kappa>0$ such that, on $M\setminus C$, the real sectional
curvature is nonnegative and
\[
  \Ricmin\ge\kappa\,\Ricmax.
\]
Then $M$ is compact.
\end{theorem}

In particular, the theorem applies when $K_g\ge0$ everywhere and the
Ricci curvature is uniformly pinched. It also covers the
positive-sectional-curvature case, since strictly positive 
sectional curvature on a K\"ahler manifold implies strictly positive
holomorphic bisectional curvature. Positivity of bisectional curvature
is assumed everywhere in Theorem~\ref{thm:main}; only the sectional
curvature and pinching conditions may be restricted to the complement
of a compact set.

The eigenvalues in \eqref{eq:ratio} are those of the Ricci endomorphism
with respect to $g$. On a K\"ahler manifold they occur in equal real pairs,
so the same ratio is obtained from the Hermitian Ricci tensor. In
particular, \eqref{eq:ratio} does not impose a uniform positive lower bound
on the individual eigenvalues. Unlike
Theorems~\ref{thm:euclidean-bk}--\ref{thm:logarithmic},
Theorem~\ref{thm:main} has no volume-growth assumption. Its pinching
hypothesis is uniform; the logarithmic weakening in
Theorem~\ref{thm:logarithmic} is not asserted for
Theorem~\ref{thm:main}.

\subsection*{Examples and sharpness}
The exponent in Theorem~\ref{thm:logarithmic} is sharp when $n=2$.
The examples below are obtained by choosing a radial profile in the
family of Wu--Zheng~\cite{WZ11} and computing the decay of the Ricci
eigenvalue ratio.

\begin{proposition}\label{prop:examples}
For every $n\ge2$ and $\delta>0$, there exists a complete
$\mathrm U(n)$-invariant K\"ahler metric on $\C^n$ with positive
bisectional curvature, Euclidean volume growth, and quadratic curvature
decay such that
\begin{equation}\label{eq:example-ratio}
  \frac{\Ricmin(x)}{\Ricmax(x)}
  \asymp\frac{1}{(\log r(x))^{1+\delta}}
\end{equation}
as $r(x)\to\infty$. Here $r$ is the distance from the origin, and
$\asymp$ means that the ratio of the two sides is bounded above and below
by positive constants.
\end{proposition}

Thus, in complex dimension two, replacing the exponent $1$ in
\eqref{eq:log-pinching} by any larger exponent makes the assertion false,
even under positive bisectional curvature. For $n\ge3$,
Proposition~\ref{prop:examples} does not determine whether the exponent
$1/(n-1)$ is optimal.

\subsection*{Conventions and organization}
We use $\rho$ to denote the Ricci form and set
$S=\Lambda_\omega\rho=\tr_\omega\rho$, where $\omega$ is the K\"ahler
form. All norms, integrals, connections, and heat operators are taken
with respect to the fixed metric. We use
\[
  dV=\frac{\omega^n}{n!},
  \qquad
  d^c=\frac{\sqrt{-1}}2(\bar\partial-\partial),
  \qquad
  \ddc=\sqrt{-1}\,\partial\bar\partial.
\]
The scalar Laplacian is
\[
  \Delta f=g^{i\bar j}\nabla_i\nabla_{\bar j}f,
\]
so $\Delta$ has nonpositive $L^2$ spectrum and is one half of the real
Laplace--Beltrami operator. We write $P_t=e^{t\Delta}$ for its heat
semigroup. The nonnegative Dolbeault Laplacian on forms is denoted by
$\Delta_{\bar\partial}$. Thus the Hodge heat equation used below is
$\eta_t+\Delta_{\bar\partial}\eta=0$; this differs from the full real
Hodge heat equation only by a constant rescaling of time. The
Riemannian scalar curvature is $2S$ in these conventions. Constants
denoted by $C$ may change from line to line.

Section~\ref{sec:euclidean} proves the Euclidean-growth results.
Section~\ref{sec:scalar} records two scalar heat estimates, and
Section~\ref{sec:forms} proves the critical-integrability vanishing theorem
and finite-mass rigidity. Section~\ref{sec:compactness} constructs the
bounded-gradient strictly plurisubharmonic exhaustion, obtains finite
top Ricci mass, and proves the compactness theorem. The examples are
constructed in Section~\ref{sec:examples}.

\section{Integral estimates and the Euclidean-growth proofs}
\label{sec:euclidean}

We first recall the integral estimate used in this section.

\begin{proposition}[\cite{Liu22,Liu24}]\label{prop:top-mass-euclidean}
Let $M^n$ be a complete noncompact K\"ahler manifold with nonnegative
Ricci curvature and Euclidean volume growth. Suppose that either the
bisectional curvature is nonnegative, or the curvature has quadratic
decay. Then
\begin{equation}\label{eq:top-mass-euclidean}
  \int_M\rho^n<\infty.
\end{equation}
\end{proposition}

The first case is the top degree case of the Cohn--Vossen estimate in
\cite{Liu22}. The second is \cite[Lemma~5]{Liu24}. These estimates do not
require any Ricci pinching assumption. For a related result without a
volume-growth assumption, Datar--Pingali--Seshadri~\cite{DPS26} proved the
finiteness of the top Ricci integral \eqref{eq:top-mass-euclidean} under
positive sectional curvature. Their general statement assumes positive
holomorphic bisectional curvature and a bounded-gradient strictly
plurisubharmonic exhaustion. Section~\ref{sec:compactness} verifies this
exhaustion hypothesis under the curvature assumptions of
Theorem~\ref{thm:main}, independently of Ricci pinching.

On the other hand, one can also prove Theorem \ref{thm:euclidean-bk} by the scalar curvature average quadratic decay 
established previously by the first author (which was proved first by the second author with 
the additional assumption that the scalar curvature is bounded) and the Li-Yau-Hamilton estimate of 
the second author. This approach of the proof avoids the above proposition completely.

The next elementary observation explains the logarithmic exponent in
Theorem~\ref{thm:logarithmic}.

\begin{lemma}\label{lem:log-integral}
Let $(M^n,\omega)$, $n\ge2$, be a complete K\"ahler manifold with
$\rho\ge0$ and $\int_M\rho^n<\infty$. If \eqref{eq:log-pinching} holds
outside a compact set, then
\begin{equation}\label{eq:log-integral}
  \int_{B(p,R)}S^n\dV=o(\log R).
\end{equation}
\end{lemma}

\begin{proof}
Let $0\le\lambda_1\le\cdots\le\lambda_n$ be the eigenvalues of $\rho$
with respect to $\omega$. At a point where $\lambda_1\ge\kappa\lambda_n$,
we have
\begin{equation}\label{eq:ratio-determinant}
  \prod_{j=1}^n\lambda_j
  \ge\kappa^{n-1}\lambda_n^n
  \ge\frac{\kappa^{n-1}}{n^n}S^n.
\end{equation}
Since $\rho^n=n!\bigl(\prod_j\lambda_j\bigr)dV$, condition
\eqref{eq:log-pinching} implies
\begin{equation}\label{eq:weighted-mass}
  S^n\,dV\le C\log(2+r)\,\rho^n
\end{equation}
outside a compact set. Fix $R_0$ sufficiently large that this set is
contained in $B(p,R_0)$. For $R>R_0$, we obtain
\begin{align*}
 \frac1{\log R}\int_{B(p,R)}S^n\dV
 &\le \frac1{\log R}\int_{B(p,R_0)}S^n\dV\\
 &\quad+C\frac{\log(2+R)}{\log R}
       \int_{M\setminus B(p,R_0)}\rho^n.
\end{align*}
First let $R\to\infty$, then let $R_0\to\infty$. The lemma follows from
the finiteness of the total mass.
\end{proof}

For the Ricci curvature case, \eqref{eq:log-integral} is precisely the
hypothesis of \cite[Corollary~7]{Liu24}. We record the corresponding
implication in the bisectional curvature case.

\begin{lemma}\label{lem:euclidean-gap}
Let $M^n$, $n\ge2$, be a complete noncompact K\"ahler manifold with
nonnegative bisectional curvature and Euclidean volume growth. If
\eqref{eq:log-integral} holds, then $M$ is flat.
\end{lemma}

\begin{proof}
By \cite[Theorems~2 and~3]{Liu24} and the existence of the asymptotic
volume ratio, the limit
\begin{equation}\label{eq:scalar-limit}
  \ell=\lim_{R\to\infty}R^{2-2n}
            \int_{B(p,R)}S\dV
\end{equation}
exists and is finite. Set
$A_j=B(p,2^{j+1})\setminus B(p,2^j)$. Condition
\eqref{eq:log-integral} gives
\[
  \sum_{j=0}^{N-1}\int_{A_j}S^n\dV=o(N).
\]
Hence there is a sequence $j_k\to\infty$ such that
$\int_{A_{j_k}}S^n\dV\to0$. Put $R_k=2^{j_k}$. By H\"older's
inequality and Bishop--Gromov volume comparison,
\begin{align*}
 R_k^{2-2n}\int_{B(p,2R_k)\setminus B(p,R_k)}S\dV
 &\le R_k^{2-2n}\Vol(B(p,2R_k))^{(n-1)/n}
       \left(\int_{A_{j_k}}S^n\dV\right)^{1/n}\\
 &\le C(n)\left(\int_{A_{j_k}}S^n\dV\right)^{1/n}
       \longrightarrow0.
\end{align*}
On the other hand, \eqref{eq:scalar-limit} implies that the left hand
side converges to $(2^{2n-2}-1)\ell$. Since $n\ge2$, we obtain
$\ell=0$. Euclidean volume growth now gives
\[
  \lim_{R\to\infty}
    \frac{R^2}{\Vol(B(p,R))}\int_{B(p,R)}S\dV=0.
\]
By the second author's optimal gap theorem~\cite{Ni12}, $M$ is flat.
\end{proof}

\begin{proof}[Proof of Theorems~\ref{thm:euclidean-bk}--\ref{thm:logarithmic}]
It suffices to prove Theorem~\ref{thm:logarithmic}, since a uniform
positive pinching constant implies \eqref{eq:log-pinching}.
Proposition~\ref{prop:top-mass-euclidean} and Lemma~\ref{lem:log-integral}
give \eqref{eq:log-integral} in both cases. If the bisectional curvature
is nonnegative, Lemma~\ref{lem:euclidean-gap} implies that $M$ is flat.
If the Ricci curvature is nonnegative and the curvature has quadratic
decay, \cite[Corollary~7]{Liu24} implies that $M$ is Ricci flat.
\end{proof}

\begin{remark}\label{rem:uniform-integrability}
Under a uniform pinching condition, \eqref{eq:ratio-determinant} and
Proposition~\ref{prop:top-mass-euclidean} give the stronger conclusion
$\int_M S^n\dV<\infty$. The logarithmic weight in
\eqref{eq:weighted-mass} is the only additional point needed for
Theorem~\ref{thm:logarithmic}.
\end{remark}

\section{Two estimates on solutions of the scalar heat equation}\label{sec:scalar}

We record the scalar estimates needed to control the boundary flux.
Their proofs require no uniform lower bound for the volume of balls.

\begin{lemma}\label{lem:scalar}
Let $(M,g)$ be a connected complete noncompact Riemannian manifold with
$\Ric_g\ge0$, and let $2\le p<\infty$. For $f\in L^p(M)$,
\begin{equation}\label{eq:scalar-estimates}
 \norm{P_tf}_{L^p}\le\norm{f}_{L^p},
 \qquad
 \norm{\nabla P_tf}_{L^p}
 \le C_p t^{-1/2}\norm{f}_{L^p}
 \quad(t>0).
\end{equation}
Moreover,
\begin{equation}\label{eq:scalar-decay}
  \lim_{t\to\infty}\norm{P_tf}_{L^p}=0.
\end{equation}
The same statements hold after a fixed positive rescaling of the scalar
Laplacian.
\end{lemma}

\begin{proof}
The heat semigroup is positivity preserving and contractive on $L^p$
(cf.\ Proposition~15.1 of \cite{Li12}).
For a bounded compactly supported smooth function $f$, the Bochner
inequality under $\Ric_g\ge0$ implies the reverse Poincar\'e estimate
\begin{equation}\label{eq:reverse-poincare}
  |\nabla P_tf|^2
  \le\frac{C}{t}\bigl(P_t(f^2)-(P_tf)^2\bigr)
  \le\frac{C}{t}P_t(f^2).
\end{equation}
One way to see the first inequality is to differentiate
\[
  s\longmapsto P_s\bigl((P_{t-s}f)^2\bigr),
  \qquad 0\le s\le t.
\]
Its derivative is a fixed positive multiple (depending on the definitions
of the Laplacian and gradient) of $P_s|\nabla P_{t-s}f|^2$. The
gradient-contraction inequality, which follows from the Bochner formula
and nonnegativity of Ricci curvature, bounds this from below by the same
multiple of $|\nabla P_tf|^2$. Integration in $s$ gives the first
estimate in \eqref{eq:reverse-poincare}.

Since $p/2\ge1$, Jensen's inequality and the $L^p$ contraction
mentioned above give
\[
 \begin{split}
  \int_M|\nabla P_tf|^p\dV
  &\le C_p t^{-p/2}
       \int_M\bigl(P_t(f^2)\bigr)^{p/2}\dV\\
  &\le C_p t^{-p/2}\norm{f}_{L^p}^{p}.
 \end{split}
\]
Approximation by compactly supported smooth functions proves the gradient
estimate for all $f\in L^p$.

We next prove \eqref{eq:scalar-decay}. First, $M$ has infinite volume.
For an elementary verification, let $\gamma$ be a minimizing ray from $o$.
The balls
\[
  B\bigl(\gamma(4^j),4^j/2\bigr),\qquad j=1,2,\ldots,
\]
are pairwise disjoint. Bishop--Gromov comparison and
$B(o,1)\subset B(\gamma(r),r+1)$ show that each of these balls has volume
bounded below by a fixed positive multiple of $\Vol B(o,1)$. Their union
therefore has infinite volume.

The $L^2$ kernel of $-\Delta$ is zero. Indeed, if $h\in L^2$ is harmonic,
standard distance cutoffs $\chi_R$ give
\[
  \int_M\chi_R^2|\nabla h|^2\dV
  \le \frac{C}{R^2}\norm{h}_{L^2}^2.
\]
Letting $R\to\infty$ shows that $h$ is constant. Infinite volume then gives
$h=0$. The spectral theorem for the nonnegative self-adjoint realization
of $-\Delta$ consequently yields
\begin{equation}\label{eq:L2-decay}
  \norm{P_th}_{L^2}\longrightarrow0
  \qquad(h\in L^2).
\end{equation}
The local parabolic mean-value estimate in Theorem~14.7 of
\cite{Li12} also gives
$\norm{P_th}_{L^\infty(K)}\longrightarrow0$ for every compact
$K\subset M$. Only the $L^2$ convergence in \eqref{eq:L2-decay} is
needed below.

If $h$ is bounded and compactly supported, then
\[
  \norm{P_th}_{L^p}^{p}
  \le \norm{h}_\infty^{p-2}\norm{P_th}_{L^2}^{2}
  \longrightarrow0.
\]
For $p=2$, this is simply \eqref{eq:L2-decay}. Density of bounded compactly
supported functions in $L^p$, together with $L^p$-contractivity, proves
\eqref{eq:scalar-decay}.
\end{proof}

\section{A critical-integrability vanishing theorem}\label{sec:forms}

Recall that a K\"ahler manifold has \emph{nonnegative quadratic orthogonal
bisectional curvature} if, at every point and for every unitary frame
$e_1,\ldots,e_n$,
\begin{equation}\label{eq:qobc}
  \sum_{i,j=1}^n R_{i\bar i j\bar j}(a_i-a_j)^2\ge0
  \qquad\text{for all }a_1,\ldots,a_n\in\R.
\end{equation}
Here $R_{i\bar i j\bar j}=R(e_i,\overline{e_i},e_j,\overline{e_j})$.
Nonnegative holomorphic bisectional curvature implies both
\eqref{eq:qobc} and nonnegative Ricci curvature. In this section we assume
only these latter two conditions.

We first specify the closed form evolution that will be used. This is
important because cutting off a closed $(1,1)$-form does not preserve
closedness. The curvature condition \eqref{eq:qobc} gives norm domination
for the evolution; no preservation of semipositivity is needed.

\begin{lemma}[Closed Hodge heat evolution]\label{lem:hodge}
Let $(M^n,\omega)$, $n\ge2$, be a connected complete noncompact K\"ahler
manifold with nonnegative Ricci curvature and nonnegative quadratic
orthogonal bisectional curvature. Let $\alpha$ be a smooth real
$(1,1)$-form such that
\[
  d\alpha=0,
  \qquad F:=|\alpha|\in L^n(M),
  \qquad f:=\Lambda_\omega\alpha.
\]
There is a smooth real Hodge heat evolution $\eta(x,t)$ with initial
value $\alpha$ for which
\begin{equation}\label{eq:heat-form}
  \eta_t+\Delta_{\bar\partial}\eta=0,
  \qquad d\eta(x,t)=0,
\end{equation}
\begin{equation}\label{eq:norm-domination}
  |\eta(x,t)|\le P_tF(x),
\end{equation}
and
\begin{equation}\label{eq:trace-heat}
  u(x,t):=\Lambda_\omega\eta(x,t)=P_tf(x).
\end{equation}
In particular,
\begin{equation}\label{eq:exact-evolution}
  \eta_t=\ddc u.
\end{equation}
The evolution converges locally to $\alpha$ as $t\downarrow0$.
No sign condition is imposed on $\alpha$ or asserted for $\eta(t)$.
\end{lemma}

\begin{proof}
For a fixed base point $o$, put
\[
  k_F(R)=\frac1{V(o,R)}\int_{B(o,R)}F\dV.
\]
H\"older's inequality gives
\begin{equation}\label{eq:average-growth}
  k_F(R)\le V(o,R)^{-1/n}\norm{F}_{L^n}.
\end{equation}
In particular, $k_F(R)=o(R^2)$. The Hodge--Laplace heat construction
of Tam and the second author~\cite[Theorem~4.1]{NT13}, whose curvature hypotheses are
precisely nonnegative Ricci curvature and \eqref{eq:qobc}, therefore
applies and gives a closed evolution. Reality and the domination
\eqref{eq:norm-domination} are part of that construction; see also
\cite[Lemmas~3.1 and~3.2]{NT13}. The Bochner inequality supplied by
\eqref{eq:qobc} makes the norm of a homogeneous $(1,1)$-form heat
solution a scalar heat subsolution, which is the input for this
pointwise domination.

Tracing the Hodge heat equation gives the scalar heat equation with
initial value $f$. Since $|f|\le C_nF$ and
$|u(t)|\le C_nP_tF$, scalar $L^n$-contractivity gives a uniform $L^n$
bound for $u$ on every finite time interval. Uniqueness in this class,
or \cite[Lemma~2.2]{NT13}, identifies $u$ with $P_tf$ and proves
\eqref{eq:trace-heat}.

Finally, the K\"ahler identities for a closed $(1,1)$-form give
\[
  -\Delta_{\bar\partial}\eta
  =\sqrt{-1}\,\partial\bar\partial\Lambda_\omega\eta;
\]
see \cite[Lemma~2.1]{NT13}. This proves \eqref{eq:exact-evolution}.
The local initial convergence is part of the heat construction.
\end{proof}

\begin{theorem}[Critical $L^n$ vanishing]\label{thm:critical}
Under the hypotheses of Lemma~\ref{lem:hodge}, the top-degree integral
is absolutely convergent and
\begin{equation}\label{eq:signed-vanish}
  \int_M\alpha^n=0.
\end{equation}
If, in addition, $\alpha\ge0$, then
\begin{equation}\label{eq:rank-vanish}
  \alpha^n\equiv0.
\end{equation}
In this semipositive case, $\alpha$ cannot be positive definite at any
point. For a semipositive form, the integrability hypothesis may
equivalently be written as $\Lambda_\omega\alpha\in L^n(M)$.
\end{theorem}

\begin{proof}
Let $\eta(t)$ be the evolution of Lemma~\ref{lem:hodge}, and let
$u=P_tf$. Define, also at $t=0$ with $\eta(0)=\alpha$,
\[
  \mu_t=\frac{\eta(t)^n}{n!},
  \qquad J(t)=\int_M\mu_t.
\]
The form $\mu_t$ may have either sign. If $\mu_t=q_t\,dV$, write
$|\mu_t|=|q_t|\,dV$ for its total variation. The elementary pointwise
bound for wedge products, followed by \eqref{eq:norm-domination}, gives
\begin{equation}\label{eq:mass-upper}
  |\mu_t|\le C_n|\eta(t)|^n\,dV
  \le C_n(P_tF)^n\,dV
  \qquad(t>0).
\end{equation}
The analogous bound at $t=0$ is
$|\mu_0|\le C_nF^n\,dV$. Thus all these top-degree integrals are
absolutely convergent. By Lemma~\ref{lem:scalar},
\begin{equation}\label{eq:mass-finite}
  |J(t)|\le\int_M|\mu_t|
  \le C_n\norm{P_tF}_{L^n}^{n}
  \le C_n\norm{F}_{L^n}^{n}.
\end{equation}
Moreover,
\begin{equation}\label{eq:form-gradient-bounds}
  \norm{\eta(t)}_{L^n}\le\norm{F}_{L^n},
  \qquad
  \norm{\nabla u(t)}_{L^n}
  \le C_n t^{-1/2}\norm{f}_{L^n}
  \le C_n t^{-1/2}\norm{F}_{L^n}.
\end{equation}

We now justify conservation of $J(t)$, including its initial value.
Choose compactly supported distance cutoffs $\chi_R$ with
\[
  0\le\chi_R\le1,\qquad
  \chi_R=1\text{ on }B(o,R),\qquad
  \supp\chi_R\subset B(o,2R),\qquad
  |d\chi_R|\le C/R.
\]
Lipschitz cutoffs are sufficient, and the integrations by parts below
are valid by smooth approximation on their compact supports. Fix
$0<a<b$. Since $d\eta(t)=0$ and $\eta_t=\ddc u$, we have
\begin{align}
 \int_M\chi_R(\mu_b-\mu_a)
 &=\frac1{(n-1)!}
   \int_a^b\int_M\chi_R\,\ddc u\wedge\eta^{n-1}\,dt\notag\\
 &=-\frac1{(n-1)!}
   \int_a^b\int_Md\chi_R\wedge d^cu\wedge\eta^{n-1}\,dt.
 \label{eq:flux}
\end{align}
The algebraic estimate $|\eta^{n-1}|\le C_n|\eta|^{n-1}$ does not
require a sign condition. It yields
\begin{align}
 \left|\int_M\chi_R(\mu_b-\mu_a)\right|
 &\le\frac{C_n}{R}
       \int_a^b\int_M|\nabla u|\,|\eta|^{n-1}\dV\,dt\notag\\
 &\le\frac{C_n}{R}
       \int_a^b\norm{\nabla u(t)}_{L^n}
                    \norm{\eta(t)}_{L^n}^{n-1}\,dt\notag\\
 &\le\frac{C_n\norm{F}_{L^n}^{n}}{R}
       \int_a^bt^{-1/2}\,dt\notag\\
 &\le\frac{C_n\norm{F}_{L^n}^{n}}{R}
       (\sqrt b-\sqrt a).
 \label{eq:flux-bound}
\end{align}
Here we used H\"older's inequality and
\eqref{eq:form-gradient-bounds}.

Keep $R$ fixed and let $a\downarrow0$. The local convergence of
$\eta(a)$ to $\alpha$ on the compact support of $\chi_R$ gives
\begin{equation}\label{eq:initial-flux}
  \left|\int_M\chi_R(\mu_b-\mu_0)\right|
  \le \frac{C_n\norm{F}_{L^n}^{n}}{R}\sqrt b.
\end{equation}
Since both $\mu_b$ and $\mu_0$ are absolutely integrable, dominated
convergence now allows $R\to\infty$. We conclude that
\begin{equation}\label{eq:mass-conservation}
  J(b)=J(0)=\int_M\frac{\alpha^n}{n!}
  \qquad(b>0).
\end{equation}
This order of limits uses only local convergence at initial time; no
interchange of a global integral with $a\downarrow0$ has been assumed.
The integrability of $t^{-1/2}$ near zero is what permits conservation
back to the initial form.

On the other hand, \eqref{eq:scalar-decay} and
\eqref{eq:mass-finite} give
\[
  |J(t)|\le C_n\norm{P_tF}_{L^n}^{n}
  \longrightarrow0\qquad(t\to\infty).
\]
Together with \eqref{eq:mass-conservation}, this proves
\eqref{eq:signed-vanish}. If $\alpha\ge0$, then $\alpha^n$ is a smooth
nonnegative measure, so its zero integral implies
\eqref{eq:rank-vanish}. In that case $|\alpha|$ and
$\Lambda_\omega\alpha$ are pointwise comparable up to dimensional
constants, which proves the final assertion.

Notice that we have not inferred pointwise vanishing of $\eta(t)^n$
from $J(t)=0$, since $\eta(t)$ need not be semipositive. The sign of the
\emph{initial} form is used only after its integral has been identified
with the conserved signed integral.
\end{proof}

\begin{theorem}[Finite-mass rigidity]\label{thm:finite}
Let $(M^n,\omega)$, $n\ge2$, be a connected complete noncompact K\"ahler
manifold with nonnegative Ricci curvature and nonnegative quadratic
orthogonal bisectional curvature. Let $\rho$ be its Ricci form and put
$S=\Lambda_\omega\rho$. Suppose that, for a constant $\eps>0$,
\begin{equation}\label{eq:trace-pinching}
  \rho\ge\eps S\omega,
  \qquad
  \int_M\rho^n<\infty.
\end{equation}
Then $\Rm_g\equiv0$.
\end{theorem}

\begin{proof}[Proof of Theorem~\ref{thm:finite}]
Let $\lambda_1,\ldots,\lambda_n\ge0$ be the eigenvalues of $\rho$
relative to $\omega$. The pinching condition gives
$\lambda_j\ge\eps S$ for every $j$. Therefore
\begin{equation}\label{eq:pinched-determinant}
  \eps^n S^n\,dV
  \le\frac{\rho^n}{n!}
  =\Bigl(\prod_{j=1}^n\lambda_j\Bigr)dV
  \le n^{-n}S^n\,dV.
\end{equation}
The finite-mass hypothesis implies $S\in L^n(M)$. Since $\rho\ge0$,
we also have $|\rho|\le C_nS$, so $|\rho|\in L^n(M)$. The Ricci form
is smooth, real, and closed. Theorem~\ref{thm:critical} applied to
$\alpha=\rho$ therefore gives $\rho^n\equiv0$. The first inequality
in \eqref{eq:pinched-determinant} then implies $S\equiv0$, and
nonnegative Ricci curvature gives $\rho\equiv0$.

To conclude full flatness under \eqref{eq:qobc}, fix a point $x\in M$
and a unit $(1,0)$-vector $X\in T_x^{1,0}M$. Extend $e_1=X$ to a
unitary frame and take $(a_1,a_2,\ldots,a_n)=(1,0,\ldots,0)$ in
\eqref{eq:qobc}. Writing
$H(X)=R(X,\overline X,X,\overline X)$, we obtain
\begin{equation}\label{eq:holomorphic-nonpositive}
  0\le2\sum_{j=2}^nR_{1\bar1 j\bar j}
  =2\bigl(\Ric(X,\overline X)-H(X)\bigr)
  =-2H(X).
\end{equation}
Thus $H(X)\le0$ for every unit $X$.

Let $d\sigma$ be the unitary-invariant probability measure on the
unit sphere in $T_x^{1,0}M$. The elementary moment identity
\[
  \int_{|z|=1}z_i\bar z_j z_k\bar z_l\,d\sigma(z)
  =\frac{\delta_{ij}\delta_{kl}+\delta_{il}\delta_{kj}}{n(n+1)}
\]
and the K\"ahler curvature symmetries give
\begin{equation}\label{eq:holomorphic-average}
  \int_{|X|=1}H(X)\,d\sigma(X)
  =\frac{2S(x)}{n(n+1)}=0.
\end{equation}
A continuous nonpositive function with zero average is identically
zero. Hence all holomorphic sectional curvatures vanish at $x$.
Polarization of the K\"ahler curvature tensor gives $\Rm(x)=0$.
Since $x$ was arbitrary, $\Rm_g\equiv0$.
\end{proof}

\section{Busemann exhaustion and compactness}\label{sec:compactness}

We first verify the exhaustion hypothesis in the general clause of
\cite[Theorem~1.1]{DPS26}. This is the only step for which a condition on
real sectional curvature is needed.

\begin{lemma}[Busemann heat smoothing]\label{lem:busemann}
Let $(M^n,g,J)$ be a connected complete noncompact K\"ahler manifold with
strictly positive holomorphic bisectional curvature. Assume that the
real sectional curvature is nonnegative outside a compact set.
Then $M$ admits a smooth strictly plurisubharmonic exhaustion $\psi$
such that
\begin{equation}\label{eq:busemann-gradient}
  \sup_M|\nabla\psi|\le1.
\end{equation}
\end{lemma}

\begin{proof}
Fix $o\in M$. We use the full Busemann function
\begin{equation}\label{eq:full-busemann}
  B(x)=\sup_{\gamma}B_\gamma(x),
  \qquad
  B_\gamma(x)=\lim_{s\to\infty}\bigl(s-d(x,\gamma(s))\bigr),
\end{equation}
where $\gamma$ ranges over all unit-speed minimizing rays from $o$.
In particular, $B$ is not the Busemann function of an arbitrary single
ray. It is a continuous plurisubharmonic function with Lipschitz
constant one; see \cite[Section~4]{NT03}. The exhaustion argument
invoked in \cite[proof of Theorem~4.2]{NT03} shows that $B$ is proper
when the real sectional curvature is nonnegative outside a compact
set.

Let $v(\cdot,t)=P_tB$. By \cite[Lemma~4.1(i)--(iii)]{NT03}, at each
positive time $v(\cdot,t)$ is smooth and strictly plurisubharmonic,
its gradient has norm at most one, and it is an exhaustion because
$B$ is an exhaustion. The harmless constant rescaling of heat time in
our Laplacian convention does not affect these statements. Fixing
one positive time $t_0$ and taking $\psi=v(\cdot,t_0)$ proves the
lemma.
\end{proof}

The strictness in this lemma does not require a simple-connectivity
assumption. In the parallel-kernel description of the heat
deformation by Tam and the second author, a proper nonzero complex-Hessian kernel would give a local
K\"ahler product, whose mixed bisectional curvatures vanish. Strictly
positive bisectional curvature excludes this. A wholly null Hessian
would make the exhaustion harmonic; it would attain a minimum and
hence be constant, which is impossible on a noncompact manifold.

\begin{proposition}[Finite top Ricci mass]\label{prop:top-mass-busemann}
Under the hypotheses of Lemma~\ref{lem:busemann}, the Ricci form
satisfies
\begin{equation}\label{eq:DPS}
  \int_M\rho^n<\infty.
\end{equation}
No Ricci pinching assumption is required for this conclusion.
\end{proposition}

\begin{proof}
Lemma~\ref{lem:busemann} gives a smooth strictly plurisubharmonic
exhaustion with uniformly bounded gradient. Together with positive
holomorphic bisectional curvature, these are exactly the hypotheses in
the general clause of Datar--Pingali--Seshadri
\cite[Theorem~1.1]{DPS26}. That result gives \eqref{eq:DPS}.
\end{proof}

\begin{proof}[Proof of Theorem~\ref{thm:main}]
Suppose, to obtain a contradiction, that $M$ is noncompact. Positive
holomorphic bisectional curvature implies $\rho>0$ everywhere.
The continuous function
\[
  x\longmapsto
  \frac{\lambda_{\min}(\Ric_g)(x)}
       {\lambda_{\max}(\Ric_g)(x)}
\]
has a positive minimum on the compact set in the theorem. Thus, after
decreasing the exterior pinching constant, we obtain
\eqref{eq:ratio} on all of $M$, with a constant $0<\kappa\le1$.

Write $0<\lambda_1\le\cdots\le\lambda_n$ for the Hermitian Ricci
eigenvalues. By \eqref{eq:ratio},
\[
  S=\sum_{j=1}^n\lambda_j
  \le\lambda_1+(n-1)\lambda_n
  \le\left(1+\frac{n-1}{\kappa}\right)\lambda_1.
\]
Therefore
\begin{equation}\label{eq:epsilon-kappa}
  \rho\ge\eps S\omega,
  \qquad
  \eps=\frac{\kappa}{n-1+\kappa}>0.
\end{equation}
Proposition~\ref{prop:top-mass-busemann} gives \eqref{eq:DPS}.
Nonnegative Ricci curvature and nonnegative quadratic orthogonal
bisectional curvature also follow from positive holomorphic
bisectional curvature. All the hypotheses of
Theorem~\ref{thm:finite} are now satisfied, so $\Rm_g\equiv0$.
This contradicts strictly positive holomorphic bisectional curvature.
Hence $M$ is compact.
\end{proof}

\subsection*{Remarks on the hypotheses and the argument}

\begin{remark}[Role of sectional curvature and the Busemann hypothesis]
Real sectional curvature is used only to obtain exhaustiveness in
Lemma~\ref{lem:busemann}. Consequently, one may instead assume that
the full Busemann function is exhaustive, while retaining positive
holomorphic bisectional curvature everywhere and uniform Ricci
pinching outside a compact set. More generally, a smooth strictly
plurisubharmonic exhaustion with bounded gradient is sufficient.
Positive bisectional curvature alone is not being asserted to supply
this exhaustion. In the special case $K_g\ge0$ everywhere, the full
Busemann function is convex as well as exhaustive; global convexity
is not required by the proof above.
\end{remark}

\begin{remark}
The proof uses the Hodge-Laplace heat equation deformation of a $(1,1)$-form, not a global solution
of $\ddc\phi=\rho$. In particular, it does not require the stronger
average-integrability hypothesis for the Poincar\'e--Lelong theorem in
\cite[Theorem~1.2]{NT13}. It also does not require a quantitative rate
for $\norm{P_t|\rho|}_{L^n}\to0$. Norm domination of the form evolution
and conservation of the signed top-degree integral, including its initial
value, allow qualitative $L^n$ decay to suffice. Semipositivity is used
only for the initial Ricci form, not for the evolved forms. The Busemann
heat deformation serves solely to verify the geometric finiteness input;
it is separate from the Hodge heat evolution of the Ricci form.

\end{remark}

\begin{remark}
The analytic Theorem~\ref{thm:finite} requires only nonnegative Ricci
curvature and nonnegative quadratic orthogonal bisectional curvature,
together with Ricci pinching and finite-mass assumptions. This weaker
analytic curvature hypothesis does not by itself supply a finiteness
theorem for the top Ricci integral. In
Proposition~\ref{prop:top-mass-busemann}, positive holomorphic
bisectional curvature is needed. 
\end{remark}

\begin{remark}[Complex dimension one]
The restriction $n\ge2$ cannot be omitted. The cigar metric
\[
  g_{\mathrm{cig}}
  =\frac{dx^2+dy^2}{1+x^2+y^2}
\]
on $\C$ is complete and noncompact, and its Gaussian curvature is
\[
  K=\frac{2}{1+x^2+y^2}>0.
\]
Its Ricci eigenvalue ratio is identically one, and direct integration gives
$\int_{\C}K\,dA=2\pi$. In the proof of
Theorem~\ref{thm:critical}, $n=1$ would require $L^1$ heat-semigroup decay,
which fails in general: a nonnegative integrable initial function retains
its $L^1$ mass on a stochastically complete manifold.
\end{remark}

\section{Examples}\label{sec:examples}

We prove Proposition~\ref{prop:examples} by using the
$\mathrm U(n)$-invariant metrics of Wu--Zheng~\cite{WZ11}. We recall the
formulas needed for the calculation; see also \cite[Section~5]{CLT16}.
In this section $t=|z|^2$ is a radial coordinate, not heat time.

Given a smooth function $\xi$ on $[0,\infty)$ with $\xi(0)=0$, define
\begin{equation}\label{eq:radial-functions}
  h(t)=\exp\left(-\int_0^t\frac{\xi(u)}u\,du\right),
  \qquad
  H(t)=\int_0^t h(u)\,du,
  \qquad
  f(t)=\frac{H(t)}t.
\end{equation}
The expression for $f$ is extended smoothly to $t=0$. The Hermitian
matrix
\begin{equation}\label{eq:radial-metric}
  g_{i\bar j}=f(t)\delta_{ij}+f'(t)\bar z_i z_j
\end{equation}
defines a K\"ahler metric, since $f>0$ and $(tf)'=h>0$. We use
$\omega=\sqrt{-1}\sum g_{i\bar j}\,dz_i\wedge d\bar z_j$.
With this normalization, the distance from the origin is
\begin{equation}\label{eq:radial-distance}
  r(t)=\frac1{\sqrt2}\int_0^t\sqrt{\frac{h(u)}u}\,du.
\end{equation}

At a point $(\sqrt t,0,\ldots,0)$, take the unitary frame whose first
vector is radial. The curvature tensor is determined by
\begin{equation}\label{eq:radial-curvature}
 \begin{aligned}
  A(t)&=\frac{\xi'(t)}{h(t)},\\
  B(t)&=\frac1{H(t)^2}\int_0^t\xi'(u)H(u)\,du,\\
  D(t)&=\frac2{H(t)^2}\int_0^t\xi(u)h(u)\,du.
 \end{aligned}
\end{equation}
Here $A=R_{1\bar1 1\bar1}$, $B=R_{1\bar1 i\bar i}$ for $i\ge2$, and
$D=R_{i\bar i i\bar i}=2R_{i\bar i j\bar j}$ for distinct $i,j\ge2$.
The last identity is only needed when $n\ge3$. The Ricci eigenvalues are
\begin{equation}\label{eq:radial-ricci}
  \lambda_{\mathrm{rad}}=A+(n-1)B,
  \qquad
  \lambda_{\mathrm{tan}}=B+\frac n2D,
\end{equation}
where the second has multiplicity $n-1$.

\begin{proof}[Proof of Proposition~\ref{prop:examples}]
Fix $0<a<1$ and $\delta>0$. We choose
\begin{equation}\label{eq:profile}
  \xi(t)=a\left(1-\bigl(\log(e+t)\bigr)^{-\delta}\right).
\end{equation}
Then $\xi(0)=0$, $\xi'(t)>0$, and $\xi(t)<a<1$. The Wu--Zheng
criterion implies that \eqref{eq:radial-metric} has positive bisectional
curvature. Smoothness at the origin follows from
\[
  \xi(t)=\frac{a\delta}{e}t+O(t^2).
\]
For $t\ge1$, the inequality $\xi\le a$ gives
$h(t)\ge h(1)t^{-a}$. Thus the integral in
\eqref{eq:radial-distance} diverges, and the metric is complete.

Put $b=1-a>0$. Since
\[
  \frac{th'(t)}{h(t)}=-\xi(t)\longrightarrow-a,
\]
l'H\^opital's rule gives
\begin{equation}\label{eq:radial-asymptotics}
  H(t)\sim\frac{th(t)}b,
  \qquad
  r(t)\sim\frac{\sqrt{2th(t)}}b.
\end{equation}
Here and below $u\sim v$ means $u/v\to1$. Indeed,
$(th)'=(1-\xi)h$, and differentiating $\sqrt{th}$ gives the second
limit. Also,
\begin{equation}\label{eq:log-asymptotics}
  \log h(t)=-a\log t+o(\log t),
  \qquad
  \log r(t)\sim\frac b2\log t.
\end{equation}
The determinant of \eqref{eq:radial-metric} is $f^{n-1}h$.
Integration in polar coordinates yields
\[
  \Vol(B(0,r(t)))=2^n\omega_{2n}H(t)^n.
\]
Combining this with \eqref{eq:radial-asymptotics}, we obtain
\begin{equation}\label{eq:example-avr}
  \lim_{r\to\infty}
  \frac{\Vol(B(0,r))}{\omega_{2n}r^{2n}}=b^n>0.
\end{equation}

We next compute the curvature. By \eqref{eq:profile},
\[
  \xi'(t)=\frac{a\delta}{(e+t)(\log(e+t))^{1+\delta}},
\]
and hence
\begin{equation}\label{eq:A-asymptotic}
  A(t)\sim\frac{a\delta}{th(t)(\log t)^{1+\delta}}.
\end{equation}
To estimate $B$, set $J(t)=\int_0^t\xi'(u)H(u)\,du$.
By \eqref{eq:radial-asymptotics},
\[
  J'(t)\sim\frac{a\delta}{b}
                     \frac{h(t)}{(\log t)^{1+\delta}}.
\]
Since
\[
  \frac{d}{dt}\left(\frac{th(t)}{(\log t)^{1+\delta}}\right)
  \sim\frac{bh(t)}{(\log t)^{1+\delta}},
\]
another application of l'H\^opital's rule gives
\[
  J(t)\sim\frac{a\delta}{b^2}
                     \frac{th(t)}{(\log t)^{1+\delta}}.
\]
It follows that
\begin{equation}\label{eq:BD-asymptotic}
  B(t)\sim\frac{a\delta}{th(t)(\log t)^{1+\delta}},
  \qquad
  D(t)\sim\frac{2ab}{th(t)}.
\end{equation}
For the second relation, we used $\int_0^t\xi h\sim aH(t)$,
which follows from $\xi\to a$ and $H(t)\to\infty$.

Equations~\eqref{eq:radial-ricci}, \eqref{eq:A-asymptotic},
and~\eqref{eq:BD-asymptotic} imply
\[
  \lambda_{\mathrm{rad}}(t)
      \sim\frac{na\delta}{th(t)(\log t)^{1+\delta}},
  \qquad
  \lambda_{\mathrm{tan}}(t)\sim\frac{nab}{th(t)}.
\]
In particular, $\lambda_{\mathrm{rad}}<\lambda_{\mathrm{tan}}$ for
sufficiently large $t$, and
\[
  \frac{\Ricmin}{\Ricmax}
  \sim\frac\delta b(\log t)^{-1-\delta}
  \asymp(\log r)^{-1-\delta}.
\]
Finally, \eqref{eq:radial-curvature}, \eqref{eq:radial-asymptotics},
\eqref{eq:A-asymptotic}, and~\eqref{eq:BD-asymptotic} show that
$|\Rm|=O(r^{-2})$. On a compact neighborhood of the origin the
curvature is bounded by smoothness. This proves all the assertions.
\end{proof}

\begin{remark}
The examples have one small Ricci eigenvalue and $n-1$ eigenvalues of
order $r^{-2}$. If $\kappa=\Ricmin/\Ricmax$, then
$\det_\omega\rho\asymp\kappa S^n$ in these examples. The general estimate
\eqref{eq:ratio-determinant} instead involves $\kappa^{n-1}$.
These powers agree when $n=2$, which explains why the examples establish
sharpness in that dimension.
\end{remark}

\subsection*{Acknowledgments}
The authors thank  J. Lott and J. Wang for their interests and 
discussions.

\end{document}